\documentclass{amsart} \usepackage{amsmath,amscd} 
\newtheorem{lemma}{Lemma} 
\newtheorem{prop}{Proposition} \newtheorem{theorem}{Theorem}
\newtheorem{rem}{Remark} \newtheorem{coro}{Corollary} \newtheorem{dfn}{Definition}

\title{Weakly Abundant Semigroups and Variants}
\author {P. G. Romeo and Siji Michel}

\keywords{Abundant semigroup, Weakly U-abundant semigroups, Variants}
\address{Dept. of Mathematics, CUSAT, Kerala, INDIA.}
\email{$romeo_-parackal@yahoo.com$}
\keywords{Abundant semigroup, Weakly abundant semigroup,  Variants, Natural partial order}
\begin{document}
\begin{abstract}  $S$ be a semigroup  and $\tilde{\mathcal L} \; [\tilde{\mathcal R}]$ be certain generalized Green's relations on $S$. If each $\tilde{\mathcal L} \; [\tilde{\mathcal R}]$- class of $S$ contains idempotent, then $S$ is a weakly abundant semigroup.
In this paper we discuss the variants of  weakly $U$- abundant semigroups and it is shown that the idempotent variants of a weakly $U$- abundant semigroup is weakly $U$- abundant. 
Further we also discuss the natural partial on variants of a weakly $U$-abundant semigroup.
\end{abstract}
\maketitle

\section{Preliminaries}
A set $S$ is a groupoid with respect to a binary operation if for every pair of elements $a,b\in S$ there is an element $a\cdot
b\in S$ which is the product of $a$ by $b$. A groupoid $S$ is a semigroup if the binary operation on $S$ is associative. $S$ be a semigroup and $a\in
S$ then smallest left ideal containing $a$ is $Sa\cup \{a\}$ and is called the principal left ideal generated by $a$ written as $S^1a$. The equivalence
relations $\mathcal L\,[\mathcal R]$ on $S$ is defined by $a\mathcal L b\,[a\mathcal Rb]$ if and only if $S^1a=S^1b\, [aS^1=bS^1]$. The intersection of
$\mathcal R$ and $\mathcal L$ is denoted by $\mathcal H$ and their join by $\mathcal D$. The two sided analogue of $\mathcal L$ and $\mathcal R$ is
$\mathcal J$. These equivalence relations are termed as Green's relations.  An element $a\in S$ is an idempotent if $e\cdot e=e$ and the set of all
idempotent elements in $S$ is denoted by $E(S)$. A semigroup $S$ is regular if for any $a\in S$ there exists an $x\in S$ such that $axa=a$. Note that
in a regular semigroups each $\mathcal L$ - class and each $\mathcal R$ - class contains idempotents.

One approach to study semigroups which are not regular is to consider the generalized Green's relations on semigroups. This idea was first introduced
by McAlister and Pastjin. The generalized Green's relations $\mathcal R^*\; [\mathcal L^*]$ on a semigroup $S$  is defined by $a\mathcal
R^*b\;[a\mathcal L^*b]$ if and only if $a\mathcal Rb\;[a\mathcal Lb]$ in some semigroup $\bar S$ containing $S$. Also we can define the join of
$\mathcal R^*$ and $\mathcal L^*$ by $\mathcal D^*,\quad \mathcal H^*=\mathcal R^*\cap \mathcal L^*$ and the two sided analogue by $\mathcal J^*$.

\begin{dfn} A semigroup is abundant if every $\mathcal R^*$- class and $\mathcal L^*$-class contains an idempotent. \end{dfn} 
\begin{lemma} The
following are equivalent for elements $a,b\in S$: \begin{enumerate} \item $a\mathcal R^*b$ \item for all $x,y\in S^1, \;xa=ya $ if and only if xb=yb.
\end{enumerate} 
\end{lemma} 
This condition is simplified when one of the elements is an idempotent. 
\begin{coro} Let a be an element of a semigroup S,
and let $e\in E(S)$. Then the following are equivalent: 
\begin{enumerate} \item $a\mathcal R^*e$ \item ea=a and for all $x,y\in S^1, \;xa=ya $ implies xe=ye. \end{enumerate} \end{coro}

\subsection{Weakly $U$ Abundant Semigroups} 
As a further generalization of the Green's relation $\mathcal L$ on a semigroup 
$S$, a new Green's relation is introduced in the following way. Let $E(S)$ be the set of idempotents in $S$ and suppose $U\subseteq E(S)$ be any non-empty subset of 
$E(S)$. Then the relation $\tilde{\mathcal L^U}$ defined by
the rule for any $a,b\in S, a \tilde{\mathcal L^U} b$ if and only if for all 
$e\in U$ the following hold 
$$ae=a\,\, \text{if and only if }\,\,be=b.$$ 
The relation $\tilde {\mathcal R^U}$ is defined dually, also it is easy to see that 
$\mathcal L\subseteq \mathcal L^*\subseteq \tilde{\mathcal L^U}$. In particular if 
$S$ is an abundant semigroup with $U=E(S)$ clearly 
$\mathcal L^*=\tilde{\mathcal L^U}$ and 
$\tilde {\mathcal L^U}$ and $\tilde {\mathcal R^U}$ are equivalence relations.

 Note that $\tilde {\mathcal L^U}$ and $\tilde {\mathcal R^U}$ need not be right [left] congruences and that 
 $\mathcal L\subseteq \mathcal L^*\subseteq \tilde{\mathcal L^U}$ and $\mathcal R\subseteq \mathcal R^*\subseteq \tilde{\mathcal R^U}$ with equality if $S$ is regular. 
Also analogous to $\mathcal{H}$ we have $\tilde {\mathcal H^U}=\tilde{\mathcal L^U}\cap\tilde{\mathcal R^U}$ and in a similar way $\tilde{\mathcal D^U}$ and $\tilde{\mathcal J^U}$.

\begin{dfn}A semigroup is called weakly $U$ abundant  if every $\tilde{\mathcal L^U}$ and every $\tilde{\mathcal R^U}$ class contains idempotent. \end{dfn} 

A semigroup $S$ is called fundamental if it cannot be shrunk homomorphicaly without collapsing idempotens together, that is, if $\varphi$ is a
homomorphism from $S$ then $$\varphi | E\;\text{one-one}\Rightarrow \varphi \;\text{one-one}$$ this is equivalent to $S$ having no nontrivial
idempotent separating congruence.

A mapping $\phi:S\rightarrow T$ between two semigroups S and T is a semigroup homomorphism if it preserves the semigroup multiplication. $im
\phi=\{a\phi\mid a\in S\}$ is a subsemigroup and $ker \phi =\{(a,b)\mid a\phi=b\phi\}$ is a congruence. If $S=T$ then $\phi$ is an endomorphism. $S$ is
a semigroup and $\nu$ is any congruence on $S$ then the natural mapping $$\varphi:S\rightarrow S/\nu, a\mapsto a\nu$$ is a semigroup homomorphism.

The fundamental homomorphism theorem states that if $\phi:S\rightarrow T$ is a homomorphism then $$S/ker \phi\cong im \phi$$

\section{Variants of abundant and weakly $U$ abundanr semigroups}

Varints of an abstract semigroup were first studied by John Hickey \cite{hic}. Let $S$ be a semigroup, the variant of $S$ with respect to an element $a\in S$ is a semigroup with the same underlying set $S$ with multiplication $ \star_{a}$ given by $x \star_{a} y =xay$ for all $x,y\in S$.

\begin{dfn} (Definition 1.1; cf. \cite{Hick})
Let $ S $ be a semigroup and $ a $ be an element of $ S $. An associative sandwich operation $ \star_{a} $ can be defined on $ S $ by $ x\star_{a}y = xay $ for all $ x,y\in S $. The semigroup $ (S,\star_{a}) $ is called the variant of $ S $ with respect to $ a $ and is denoted as $ S^{a} $.
\end{dfn}

Variants of full transformation semigroups, symmetric inverse monoids, regular semigroups are widely studied  (see. cf.\cite{Hick},\cite{Tsy},\cite{law}).
In the following we consider variants of abundant semigroup $S$ and discuss the 
Green's relations $\mathcal L^{*a},\,\mathcal R^{*a}\,
\mathcal H^{*a}$ and $\mathcal D^{*a}$ on the variant $S^a$ of $S$.

\begin{dfn} Let $S$ be an abundant semigroup. $S^a,\, a\in S$ is the variant with respect to $a$, for $x,y\in S^a$ define the Green's relations of $S^a$ as 
\begin{align*}
 x\mathcal R^{*a}y&\Leftrightarrow  u\star_{a}x=v\star_{a}x \Leftrightarrow u\star_{a}y=v\star_{a}y\\
 x\mathcal L^{*a}y&\Leftrightarrow  x\star_{a}u=x\star_{a}v \Leftrightarrow y\star_{a}u=y\star_{a}v ,\,\text{for all}\,u,v\in S^a\\
 \mathcal H^{*a}&=\mathcal R^{*a}\cap\mathcal L^{*a},\, \text{and\,} \,\mathcal D^{*a}=
 \mathcal R^{*a}\circ  \mathcal L^{*a}.
\end{align*}
\end{dfn}

\begin{prop} Let $S$ be an abundant semigroup. Consider the following sets 
\begin{align*}
 P_{1}&=\{ x \in S : ax \, \mathcal{R}^{*} \, x \} \\
 P_{2}&=\{  x\in S : xa \, \mathcal{L}^{*} \, x \} \\
  P&=P_{1} \cap P_{2} 
\end{align*}
For $ x \in S^a $, the equivalence classes $\mathcal R^{*a},\,\mathcal L
^{*a}$ of the variant $S^a$ are as follows:
$$R_{x}^{*a}\cap P_1 = R_{x}^{*}\quad \text{and} \quad L_{x}^{*a}\cap P_2 = L_{x}^{*}$$		
\end{prop}

\begin{proof} Suppose $ x R^{*a}y$ then for every $u,v\in S^a\,,\quad u\star_{a}x=v\star_{a}x \Leftrightarrow u\star_{a}y=v\star_{a}y$ that is $uax=vax \Leftrightarrow uay=vay$. Since $x,y\in P_1$, we have $uax=vax\Leftrightarrow ux=vx$, and so  
$$ux=vy\Leftrightarrow uax=vax\Leftrightarrow uay=vay\Leftrightarrow uy=vy$$
thus $x\mathcal R^* y$. Dually the proof $L_{x}^{*a}\cap P_2 = L_{x}^{*}$ follows.
\end{proof}

\begin{prop}
Let $ S $ be an abundant monoid. If $ a \in S $ is invertible then $ S^{a} $ is abundant.
\end{prop}	
\begin{proof}
For every $a\in S$ we have $e,f\in E(S)$ such that 
$e \mathcal{R}^{*} a\mathcal{L}^{*} f$. Since $ a $ be invertible, then there exist 
$ a^{-1}\in S $ such that $ aa^{-1}= 1= a^{-1} a $ and for $a^{-1}\in S^a,\,\,ea^{-1}$ and $ a^{-1}f$ are in $E(S^a)$. Further $ea^{-1}\star_a a^{-1}=ea^{-1}$, and for every 
$u,v\in S^a$, we have
$$u\star_aa^{-1}=v\star_aa^{-1}\Rightarrow uaa^{-1}=vaa^{-1}\Rightarrow u=v$$
and so $uaea^{-1}=vaea^{-1}\Rightarrow u\star_aea^{-1}=v\star_aea^{-1}$. ie., 
$ea^{-1}\mathcal{R}^{*} a^{-1}$. 

Similarly it can be seen that $a^{-1}\mathcal{L}^{*}a^{-1}f$, thus $S^a$ is abundant.
\end{proof}

\begin{dfn} Let $S$ be a semigroup, $e\in E(S)$. The variants with respect to idempotents $e\in E(S)$ are semigroups $S^e$ with binary composition $\star_e$ defined by $a\star_eb=aeb$ and is given by 
$$S^e=\{a,b\in S\,\mid a\star_eb=aeb, \,\,\text{with}\,\, e\in E(S)\}.$$
\end{dfn}
 
\begin{theorem}
Let $S$ be a weakly $U$ abundant semigroup. Then the variants with respect to idempotents $e\in U$ is $\{S^e,\,\,e\in U\subseteq E(S)\}$ called the idempotent variants  are weakly $U$ abundant.
\end{theorem}
\begin{proof} Since $S$ is weakly abundant, every $\tilde{\mathcal L^U}$ and every 
$\tilde{\mathcal R^U}$ class contains idempotent. $a\tilde{\mathcal L^U}b$ implies that forall $e\in U,\,\,ae=a \,\, \text{if and only if }\,\,be=b.$ For $e\in U$, 
clearly $e\in E(S^e)$ and so
$$a\star_e e=aee=ae=a\Leftrightarrow b\star_ee=bee=be=b$$
Similarly $e\star_e a=eea=ea=a\Leftrightarrow e\star_eb=eeb=eb=b$ ie., $S^e$ is weakly $U$ abundant.
\end{proof}

\section{Congruences on variants}	
An equivalence relation $\rho$ on a semigroup $S$ is a congruence on $S$ 
 if and only if
 $$(xs, yt) \in \rho$$
for all pairs $(x, y), (s, t) \in \rho$.

\begin{dfn} Let $S$ be a semigroup, and let $\rho$ be a congruence on $S$. The quotient semigroup $S/\rho$ is the semigroup whose elements are the congruence classes of $\rho$, and whose operation $\ast$ is defined by
$$ [a]_{\rho} \ast [b]_{\rho} = [ab]_{\rho}$$ 
for $a, b\in S.$
\end{dfn}

Let $ S $ be a semigroup. For each $ u\in S $ we define a relation $ \lambda^{u} $ on $ S $ by 
$$ x  \, \lambda^{u} \, y \quad \Leftrightarrow \quad  ux = uy $$
and the relation $ \rho^{u} $ on $ S $ by 
$$ x \rho^{u} y \quad \Leftrightarrow \quad xu=yu. $$ 
These relations were among the congruences introduced in the context of sandwich semigroups by Symons(cf.\cite{sym}).

\begin{lemma} $S$ is a semigroup, $u\in S$ and $S^u$ the variant of $S$ with respect to $u$. Then the relations $\lambda^u\,\,[\rho^u]$ defined by
$$x \, \lambda^{u} \, y \quad \Leftrightarrow \quad  ux = uy\quad \text{and}$$
$$x\, \rho^{u}y\,\quad \Leftrightarrow \quad xu=yu$$
is a left [right] congruence on $ S^{u} $.
\end{lemma}

\begin{proof} As $ \lambda^{u} $ defined on $S$ by $ x  \, \lambda^{u} \, y \,
\Leftrightarrow \,  ux = uy $ is a congruence, we have $(x,y)\in \lambda^{u}$ implies $(px,py)\in \lambda^{u}$ for every $p\in S$. Now regarding $ \lambda^{u} $ as a the relation  on 
$ S^u $ defined by 
\begin{align*}
   x  \, \lambda^{u} \, y &\Leftrightarrow u \star_{u} x=u \star_{u} y\\
   &\Leftrightarrow  uux=uuy
\end{align*}
 Now for $ p\in S^u $, we have
 \begin{align*}
   uux=uuy &\Leftrightarrow puux=puuy\\
   &\Leftrightarrow p\star_{u} ux=p\star_{u} uy\\
   &\Leftrightarrow (px,py)\in \lambda^u .
 \end{align*}  
Hence $\lambda^u$ if a left congruence. Similarly it is easily seen that $\rho^{u}$ is a right congruence on $ S^{u} $.
\end{proof}

\begin{lemma}
$ S $ be a semigroup,  $ u\in S ,\,\,S^u$ the variant wit respect to $u$ and $\lambda^u$ the congruence defined on $S^u$.  Then $ S^{u}/\lambda^{u}\cong uS. $
\end{lemma}
\begin{proof}
	Define a map $ \phi :S^u \to uS $ by $ \phi(x)= ux.$ Then $ \phi $ is a homomorphism, for $ x,y \in S $, $ \phi (x \star_{u} y)= u(xuy)
= (ux)(uy)= \phi(x) \phi(y) $. Now  
$$ ker\,(\phi^{-1} \circ \phi)=\{(x,y)\,:\, ux=uy \}$$
ie., $ker\,(\phi^{-1} \circ \phi)=\lambda^u$, hence we have $S^u/\lambda^u\cong uS$.
\end{proof}

\begin{coro}
	Let $ S $ be a semigroup and $ a , b $ be $ \mathcal{L} $- related elements of $ S $. Then $ S^{a}/\lambda^{a}\cong S^{b}/\lambda^{b} $.
\end{coro}

\section{partial order on variants}

In any semigroup $(S,\cdot)$ for $a,b\in S$  the relation $a\leq b$ if and only if $a=xb=by,\,\,xa=a$ for some $x,y\in S^1$ is called the natural partial order of $S$. When restricted to the idempotent sets $E(S)$ this coinsides with te usual ordering $e\leq f$ if and only if $ef=fe=e$.

 \begin{lemma} Let $\{S^e\, \mid e\in E(S)\}$ be the idempotent variant of a semigroup $S$. Then the set of idempotents in $S^e$ is
 $$E(S^e)=\{f\in E(S)\,:\,f\leq e \}.$$
 \end{lemma}
 
\begin{proof} Let $f\in E(S)$ and $f\leq e$. Then 
 $$f \star_{e}f=fef=ff=f$$
 ie., $f\in E(S^e)$. Conversly if $f\in E(S^e)$ then
 $$f \star_{e} f= f \Leftrightarrow fef=f$$
 hence $f^2=f$ and $fe=ef=f$ that is $f\leq e$ 
\end{proof} 
 
Now define a relation on $ E(S^{e} )$ by, 
$$ x\leq_{e} y \Leftrightarrow  x \star_{e} y =y \star_{e} x =x $$  
for all $ x,y \in E(S^e).$

\begin{lemma}
	The relation $ \leq_{e} $ defined on $ E(S^{e} )$ is a partial order on $ E(S^{e}) $.
\end{lemma}
\begin{proof}
	For $ x \in E(S^{e}) $, we have $ x\leq_{e} x =xex=x$ since $xe=ex=x$. ie., $x\leq_ex$. 
Consider $ y,z$ also in  $E(S^{e}) $, with $ x\leq_{e} y  $ and $ y\leq_{e} z$, then 
$x \star_{e} y=y\star_{e}x=x$ and $y \star_{e} z=z\star_{e}y=y$
$$x=x \leq_{e} y=x \star_{e}(y \star_{e} z)=(x \star_{e}y)\star_{e} z=x \star_{e}z$$
and similarly $z \star_{e}x=x$. Also when  $ x\leq_{e} y$ and $ y\leq_{e} x$ then obviously $x=y$. Hence $ \leq_{e} $ is a partial oreder on $ E(S^{e}) $.  
\end{proof}

Next we generalise this partial order in $ E(S^{e}) $ to the whole of $ S^{e} $ by defining  
$ a\leq_{e} b$ if and only if there exists $ x,y \in (S^{e})^{1} $ such that 
$$ x\star_{e} b= b \star_{e} y=a \quad\text{and}\quad  x\star_{e} a=a $$.
\begin{rem}
	For elements $ a,b \in S^{e}, \,\, a\leq_{e} b \Rightarrow a\leq b $ in $ S $.
\end{rem}
\begin{proof}
	Suppose $ a\leq_{e} b $, then 
$$ x\star_{e} b= b \star_{e} y=a \quad\text{and}\quad  x\star_{e} a=a $$	
for some $ x,y \in (S^{e})^{1} $. ie., $ a=xeb = bey$ and $a=xea$, thus 
$a= x'b= b y'$ and $a= x'a =  ay' $ for some $ x',y' \in S^{1} $, hence $ a\leq b $ in $ S $.
\end{proof}

\begin{lemma}
$ S $ is a semigroup $e\in E(S)$ and $S^e$ be an idempotent variant. For $a,b\in S^e$, then the following hold.\begin{enumerate}
		\item $ a\leq_{e} f,$ and $ f\in E(S^{e}) $ then $ a\in E(S^{e}) $
		\item $ a\leq_{e} b $, and $ b $ is regular element then $  a $ is regular
			\end{enumerate}
\end{lemma}
\begin{proof}
	\begin{enumerate}
\item For $ a\leq_{e} f$, then there exists $x,y\in (S^e)^1$ such that 
 $a= x\star_{e}f= f\star_{e} y$ and $a = x\star_{e} a$. 
 Now $ a^{2}= a\star_{e} a = x\star_{e}a \star_{e}f \star_{e} y 
 = x\star_{e} a \star_{e} y = a\star_{e} y = a $. ie., $ a\in E(S^{u}) $.\\

\item For $ a\leq_{e} b$ we have $x,y\in (S^e)^1$ such that  $a= x\star_{e} b = b\star_{e} y$ and $a =x\star_{e} a $. Since $ b $ is regular in $ S^{e} $ there exist $ b' $ in $ S^{e} $ such that $ b=b\star_{e} b' \star_{e} b $. Then for every inverse $ b' $ of $ b $ in $ S^{e} $,
		$ a= a\star_{e} y = (x\star_{e} b)\star_{e} y = x\star_{e} (b\star_{e} b' \star_{e} b)\star_{e} y = (x\star_{e} b) \star_{e} b' \star_{e} (b\star_{e} y) = a\star_{e} b' \star_{e} a $. Therefore, $ a $ is regular in $ S^{e} $.\\
	
\end{enumerate}
\end{proof}

\end{document}